\RequirePackage[l2tabu, orthodox]{nag}
\documentclass[a4paper,11pt]{article}
\usepackage[margin=4.1cm]{geometry}
\usepackage{mathrsfs}
\usepackage{lmodern}
\usepackage{mparhack}
\usepackage{amsmath, amssymb, amsfonts, amsthm, mathtools}
\usepackage{fixmath}
\mathtoolsset{centercolon}
\usepackage[latin1]{inputenc}
\usepackage[T1]{fontenc}
\usepackage{bookmark}%ArXiv compiles better with this and not hyperref

\usepackage{enumitem}
\newlist{casesp}{enumerate}{4} %% new list environment based on enumerate with 
\setlist[casesp]{align=left, %% alignment of labels
                 listparindent=\parindent, %% same indentation as in normal text
                 parsep=\parskip, %% same parskip as in normal text
                 font=\normalfont\bfseries, %% font used for labels
                 leftmargin=0pt, %% total amount by which text is indented
                 labelwidth=0pt, %% width of labels (=how much they stick out 
                 itemindent=.4em,labelsep=.4em, %% space between label and text
                 partopsep=0pt, %% extra vertical space above and below if 
                 itemsep=0pt, %% vertical space after each item
                 }
\setlist[casesp,1]{label=Case~\arabic*:,ref=\arabic*}
\setlist[casesp,2]{label=Case~\thecasespi.\arabic*:,ref=\thecasespi.\arabic*}
\setlist[casesp,3]{label=Case~\thecasespii.\arabic*:,ref=\thecasespii.\arabic*}
\setlist[casesp,4]{label=Case~\thecasespiii.\arabic*:,ref=\thecasespiii.\arabic*
}

\theoremstyle{plain}
\newtheorem{theorem}{Theorem}
\newtheorem{lemma}[theorem]{Lemma}
\newtheorem{proposition}[theorem]{Proposition}

\theoremstyle{definition}
\newtheorem{example}[theorem]{Example}

\newcommand{\setbuilder}[2]{\left\{#1\;:\;#2\right\}}
\newcommand{\myangle}{\sphericalangle}
\newcommand{\ark}[2]{\myangle(#1,#2)}
\newcommand{\numbersystem}[1]{\mathbb{#1}}
\newcommand{\bR}{\numbersystem{R}}
\newcommand{\bZ}{\numbersystem{Z}}
\DeclareMathOperator{\conv}{conv}
\DeclareMathOperator{\supp}{supp}
\DeclareMathOperator{\bd}{bd}
\renewcommand{\Re}{\bR}

\newcommand{\norm}[2][]{\left\lVert#2\right\rVert_{#1}}

\author{M\'arton Nasz\'odi\footnote{Alfr\'ed R\'enyi Inst. of Math. and Dept. of Geometry, E\"otv\"os Lor\'and University, Budapest}
\and Vilmos Prokaj\footnote{Department of Probability Theory and Statistics, E\"otv\"os Lor\'and University, Budapest}
\and Konrad Swanepoel\footnote{Department of Mathematics, London School of Economics}}
\title{Angular measures and Birkhoff orthogonality in Minkowski planes\footnote{Part of the research was carried out while MN was a member of 
J\'anos Pach's chair of DCG at EPFL, Lausanne, which was supported by Swiss 
National Science Foundation Grants 200020-162884 and 200021-175977, and while MN and KS visited the Mathematical Research Institute Oberwolfach in their Research in Pairs programme.
MN also acknowledges the support of the National Research, Development and 
Innovation Fund grant K119670.}}\date{}

\begin{document}
\maketitle
\begin{abstract}
Let $x$ and $y$ be two unit vectors in a normed plane $\Re^2$. We say that $x$ 
is \emph{Birkhoff orthogonal} to $y$ if the line through $x$ in the direction 
$y$ supports the unit disc. A \emph{B-measure} (Fankh\"anel 2011) is an angular 
measure $\mu$ on the 
unit circle for which $\mu(C)=\pi/2$ whenever $C$ is a shorter arc of the unit 
circle connecting two Birkhoff orthogonal points. We present a characterization 
of the normed planes that admit a B-measure.
\end{abstract}

\section{Introduction}

Let $K$ be an $o$-symmetric convex body on the plane, and consider 
the normed plane $(\Re^2,\norm[K]{\cdot})$, where 
$\norm[K]{x}=\min\setbuilder{\lambda>0}{x\in \lambda K}$ for any $x\in\Re^2$.

Let $x$ and $y$ be two non-zero vectors in $\Re^2$. We say that $x$ is 
\emph{Birkhoff orthogonal} to $y$, and denote it by $x\dashv y$, if 
$\norm[K]{x}\leq\norm[K]{x+ty}$ for all $t\in\Re$. Geometrically, it means that 
the line through $x$ in the direction $y$ supports $\norm[K]{x}K$. In 
general, Birkhoff orthogonality is not a symmetric relation.
Normed planes where Birkhoff orthogonality is symmetric are called \emph{Radon 
planes} and the boundaries of their unit balls \emph{Radon curves}
(see the survey \cite{MS06}).

A Borel measure $\mu$ on $\bd K$ is called an \emph{angular measure}, if 
$\mu(\bd K)=2\pi$, $\mu(X)=\mu(-X)$ for every Borel subset $X$ of $\bd K$, and 
$\mu$ is \emph{continuous}, that is, $\mu(\{x\})=0$ for every $x\in\bd K$. 
Clearly, for any $K$, there is an angular measure, as the one-dimensional 
Lebesgue measure on $\bd K$ normalized to $2\pi$ is one, but has no relation to 
the geometry of $(\Re^2,\norm[K]{\cdot})$.
A natural problem then is to find angular measures with interesting geometric 
properties.
For instance, Brass \cite{Br96} showed that whenever the unit ball is not a 
parallelogram, there is an angular measure in which the angles of any 
equilateral triangle are equal.
This type of angular measure is very useful in studying packings of unit balls 
\cite{Br96, S18}.
Angular measure with other properties have been proposed; see the survey 
\cite[Section~4]{BHMT17} for an overview.
An angular measure $\mu$ is called a \emph{B-measure} \cite{Fa11} if 
$\mu(C)=\pi/2$ for 
every closed arc $C$ of $\bd K$ that contains no opposite points of $\bd K$, 
and whose end points $x$ and $y$ satisfy $x\dashv y$.

The main result, Theorem~\ref{thm:bmeasureexists}, of this note is a 
characterization of the normed planes $(\Re^2,\norm[K]{\cdot})$ which admit a 
B-measure.

We call a point $x$ in $\bd K$ an \emph{Auerbach point}, if there is a $y\in\bd 
K$ such that $x\dashv y$ and $y\dashv x$. 
It is well known that Auerbach points exist for any norm \cite[Section 
3.2]{T96}.
We denote the set of Auerbach points of $K$ by $A(K)$.
Note that $A(K)$ is a closed subset of $\bd K$.
We denote the union of 
open non-degenerate line segments contained in $\bd K$ by $E(K)$.

\begin{theorem}\label{thm:bmeasureexists}
 Let $K$ be an origin-symmetric convex body in $\Re^2$. Then there is a 
B-measure on $\bd K$ if, and only if, the set $A(K)\setminus E(K)$ 
is uncountable.
\end{theorem}

This result strengthens Fankh\"{a}nel's \cite[Theorem~1]{Fa11}, where the 
existence of a B-measure is shown under the condition that $A(K)\setminus E(K)$ 
contains an 
arc.
(Fankh\"anel does not explicitly exclude line segments, but it is clear that 
they have to be excluded, as line segments in $A(K)$ necessarily have measure 
$0$ for any B-measure; see Lemma~\ref{lem:suppina}.)
We prove Theorem~\ref{thm:bmeasureexists} in Section~\ref{sec:bmeasureexists}, 
where we also present a smooth, strictly convex, centrally symmetric planar 
body $K$ such 
that $A(K)$ is the union of two disjoint copies of the Cantor set and a 
countable set of isolated points (Example~\ref{ex:auerbachcantor}).
Thus, $A(K)$ is of Lebesgue measure zero and 
yet, by Theorem~\ref{thm:bmeasureexists}, there is a B-measure on $\bd K$.

We recall that a subset of a topological space is called \emph{perfect}, if it 
is closed and 
has no isolated point.
Recall that the \emph{support} $\supp(\mu)$ of a Borel measure $\mu$ on a 
topological space $X$ is the set of all $x\in X$ such that all open sets 
containing $x$ have positive $\mu$-measure.
In the proof of Theorem~\ref{thm:bmeasureexists}, we 
rely on the following result.
\begin{proposition}\label{thm:continuoumeasure}
Let $H\subset[0,1]$ be a non-empty, closed, perfect set. Then there is a 
continuous probability measure on $[0,1]$ whose support is $H$.
\end{proposition}
This is a well-known result holding more generally for any separable complete 
metric 
space \cite[Chapter II, Theorem~8.1]{P2005}, but for the convenience of the 
reader we 
present an explicit construction in Section~\ref{sec:3}.
It is easy to see that the converse of Proposition~\ref{thm:continuoumeasure} 
holds, namely that the support of any continuous measure on 
$[0,1]$ is a perfect set.
It is well known that every non-empty perfect set is uncountable 
\cite[Theorem~2.43]{Ru76} 
and every uncountable closed 
set contains a perfect set \cite[Section~6B]{Ke95}.
(More generally, all Borel sets and analytic sets have this property 
\cite{Ke95}, but we will only need it for $F_\sigma$ sets).

\section{The Auerbach set and B-measure}\label{sec:bmeasureexists}
Given two non-opposite points $a,b\in \bd K$, we denote by $\ark{a}{b}$ the 
closed arc from $a$ to $b$ that does not contain any opposite pairs of points.
We denote the closed line segment with endpoints $a,b\in\Re^2$ by $[a,b]$.

\begin{lemma}\label{lem:suppina}
 Let $K$ be an origin-symmetric convex body in $\Re^2$ and $\mu$ be a B-measure 
on $\bd K$. Then $\supp(\mu)\subseteq A(K)\setminus E(K)$.
\end{lemma}

\begin{proof}
Consider a non-degenerate line segment $[x^-,x^+]\subset\bd K$.
Let $y\in\bd K$ parallel to $[x^-,x^+]$.
Since $x^-,x^+\dashv y$, we have $\mu([x^+,y])=\mu([x^-,y])=\pi/2$. Thus, 
$\mu([x^-,x^+])=0$, and hence, no 
$x$ strictly between $x^-$ and $x^+$ is in $\supp(\mu)$.

Next, let $x$ be a point in $\bd K\setminus A(K)$.
Let $y_1,y_2\in\bd K$ such that $x\dashv y_1$ and $y_2\dashv x$.
Then $y_1\neq y_2$.
By possibly replacing $y_2$ by $-y_2$, we assume without loss of generality 
that $y_1$ and $y_2$ are in the same open half plane bounded by the line $ox$.
By possibly replacing $x$ by $-x$, we may also assume without
loss that $y_2$ and 
$x$ are in the same open half plane bounded by $oy_1$.
Let $x_1$ and $x_2$ be points on the same side of $oy_1$ as $x$ such that 
$y_1\dashv x_1$ and $x_2\dashv y_2$.
Then $x_1,x_2\neq x$.
Because $y_2$ is between $x$ and $y_1$, we have that $x_1$ and $x_2$ are in 
opposite open half planes bounded by $ox$.
Since $\mu$ is a B-measure, $\mu(\ark{x}{x_1})=\mu(\ark{x}{x_2})=0$.
Thus, $\mu(\ark{x_1}{x_2})=0$, hence $x\notin\supp(\mu)$.
\end{proof}

\begin{proof}[Proof of Theorem~\ref{thm:bmeasureexists}]
Let $\mu$ be a B-measure. Since $\supp(\mu)$ is a perfect set, hence 
uncountable, Lemma~\ref{lem:suppina} gives that $A(K)\setminus E(K)$ is 
uncountable.

Conversely, assume that $A(K)\setminus E(K)$ is uncountable.
Define a map $\phi:A(K)\setminus E(K)\to \bd K$ by setting $\phi(x)$ to be the 
first $y\in A(K)$ in the positive direction along $\bd K$ from $x$ such that 
$x\dashv y$ and $y\dashv x$.
Then $\phi$ is monotone, but not necessarily injective.
However, if $\phi(x_1)=\phi(x_2)$, then $x_1\dashv y$ and $x_2\dashv y$, hence 
$[x_1,x_2]$ is a line segment on $\bd K$.
It follows that for any given $y\in \bd K$, there are at most two values of 
$x\in A(K)\setminus E(K)$ such that $\phi(x)=y$, and there are at most 
countably many $y\in\bd K$ for which there is more than one $x\in A(K)\setminus 
E(K)$ such that $\phi(x)=y$.
Thus, $\phi$ is a Borel measurable map.

For any $x\in\bd K$, let $x^+$ denote the first element of $A(K)\setminus E(K)$ 
in the positive direction from $x$, and similarly, let $x^-$ be the first 
element of $A(K)\setminus E(K)$ in the negative direction from $x$.

Fix $a\in A(K)\setminus E(K)$.
Let $b=\phi(a)$ and $a'=\phi(b^+)$.
If $b^+\neq b$ then $a'$ is not strictly between $-a$ and $b$.
If $b^+=b$ and $a'$ is strictly between $-a$ and $b$, then $[-a,a']$ is a line 
segment on $\bd K$.
In either case, we obtain that $\ark{a}{b}\cap A(K)\setminus E(K)$ or 
$\ark{b^+}{a}\cap A(K)\setminus E(K)$ is uncountable.
Thus we may assume without loss of generality that $\ark{a}{b}\cap 
A(K)\setminus E(K)$ is uncountable.
It then contains a perfect set $H$ and there is a continuous probability 
measure $\nu$ on the Borel sets of $\bd K$ with support $H$, of which the 
existence is guaranteed by 
Proposition~\ref{thm:continuoumeasure}.
We use $\nu$ to define the B-measure as follows.
For any Borel set $A\subset \bd K$, let 
\begin{equation}
\mu(A)=\frac{\pi}{2}\big[\nu(A)+\nu(-A)+\nu(\phi^{-1}(A))+\nu(\phi^{-1}(-A))\big
].
\end{equation}
Then $\mu$ is clearly an angular measure.
To see that $\mu$ is a B-measure, let $x, y\in \bd 
K$ with $x\dashv y$.
By possibly replacing $x$ by $-x$ and $y$ by $-y$, we may assume that 
$x\in\ark{a}{b}$ or $x\in\ark{b}{-a}$, and that $y\in\ark{a}{b}\cup\ark{b}{-a}$.

\begin{casesp}
    \item $x\in\ark{a}{b}$.
    Then either $y\in\ark{a}{b}$ or $y\in\ark{b}{-a}\setminus\{b\}$.
    \begin{casesp}
        \item $y\in\ark{a}{b}$. \label{1.1}
        Then there are two cases depending on the relative position of $x$ and 
$y$.
        \begin{casesp}
            \item $x\in\ark{a}{y}$. \label{1.1.1}
                In this case necessarily $y=b$ or $x=a$.
                If $y=b$ then $[a,x]$ is a segment on $\bd K$, hence 
$\mu(\ark{a}{x})=0$ and $\mu(\ark{x}{y})=\pi/2$.
                If $x=a$ and $y\neq b$, then we show that $\ark{y}{b}\cap A(K)$ 
is finite.
                Suppose there exist $y'\in\ark{y}{b}\cap A(K)\setminus\{b,y\}$.
                Then there exist $a'\in\ark{a}{-b}\setminus\{a\}$ such that 
$a'\dashv y'$, $y'\dashv a'$.
                However, since $a\dashv y$ and $a\dashv b$, we obtain $a\dashv 
y'$,
                hence $[a',a]$ is a line segment on $\bd K$ and $y'=y$.
                Therefore, $\mu(\ark{y}{b})=0$.

            \item $x\in\ark{y}{b}$.
                In this case necessarily $x=b$ or $y=a$.
                This case is finished in a similar way as Case~\ref{1.1.1}.
        \end{casesp}
        \item $y\in\ark{b}{-a}\setminus\{b\}$.
        We have to show that $\phi^{-1}(\ark{b}{y})=\ark{a}{x}\cap 
A(K)\setminus E(K)$ up to measure $0$.
        \begin{casesp}
            \item $x\notin A(K)\setminus E(K)$.
                Then $\phi(x^-)\in\ark{b}{y}$ and $\phi(x^+)\in\ark{y}{-a}$.
                Note that $\phi^{-1}(\ark{b}{\phi(x^-)}) = \ark{a}{x}\cap 
A(K)\setminus E(K) \cup\phi^{-1}(\{\phi(x^-),b\})$, where 
$\phi^{-1}(\{\phi(x^-),b\})$ contains at most $4$ points.
                Similarly, $\phi^{-1}(\ark{b}{\phi(x^+)}) = \ark{a}{x}\cap 
A(K)\setminus E(K) \cup\phi^{-1}(\{\phi(x^+),b\})$, and it follows that 
$\phi^{-1}(\ark{b}{y})=\ark{a}{x}\cap A(K)\setminus E(K)$ up to $2$ points.

            \item $x\in A(K)\setminus E(K)$.
                We show that $\ark{\phi(x)}{y}\cap A(K)$ is finite, from which 
follows that $\phi^{-1}(\ark{b}{y})$ equals $\ark{a}{x}$ up to finitely many 
points.

                Suppose that there exists $y'\in\ark{\phi(x)}{y}\cap 
A(K)\setminus\{\phi(x),y\}$.
                \begin{casesp}
                    \item $y\in\ark{b}{\phi(x)}$. \label{1.2.2.1}
                    Then there exists $x'\in\ark{a}{x}$ such that $x'\dashv 
y'$, $y'\dashv x'$.
                    Since $x\dashv\phi(x)$ and $x\dashv y$, we have $x\dashv 
y'$.
                    It follows that $[x,x']$ is a segment on $\bd K$ parallel 
to $y'$, but then $y=y'$ unless $x=x'$, but then $\phi(x)=y'$.
                    \item $\phi(x)\in\ark{b}{y}$.
                    Similar to Case~\ref{1.2.2.1} we obtain that $x=x'$, and 
then $\phi(x)\dashv x$ and $y'\dashv x$, hence $[\phi(x),y']$ is a segment on 
$\bd K$ parallel to $x$.
                    This holds for all $y'\in 
A(K)\cap\ark{\phi(x)}{y}\setminus\{\phi(x),y\}$.
                    It follows that 
$\phi^{-1}(\ark{\phi(x)}{y})\subseteq\phi^{-1}(\{\phi(x),y\})$, which consists 
of at most $4$ points.
                \end{casesp}
        \end{casesp}
    \end{casesp}
\item $x\in\ark{b}{-a}\setminus\{b,-a\}$.
    If $y\in\ark{b}{-a}$, then as in Case~\ref{1.1}, $y=b$ or $y=-a$.
    But then, by possibly replacing $y$ with $-y$, we can assume that 
$y\in\ark{a}{b}$.
    Thus without loss of generality, $y\in\ark{a}{b}$.
    Then $\phi(y^+)\in\ark{x}{-a}$ and $\phi(y^-)\in\ark{b}{x}$.
    Then $\phi^{-1}(\ark{b}{\phi(y)}$, which equals $\ark{a}{y^-}\cap 
A(K)\setminus E(K)$ up to $2$ points, is contained in $\phi^{-1}\ark{b}{x})$, 
which is in turn contained in $\phi^{-1}(\ark{b}{\phi(y^+)})$, which equals 
$\ark{a}{y^+}\cap A(K)\setminus E(K)$ up to $2$ points.
    It follows that $\nu(\phi^{-1}(\ark{b}{x})=\nu(\ark{a}{y})$.
\end{casesp}

This completes the proof of Theorem~\ref{thm:bmeasureexists}.
\end{proof}

\begin{example}\label{ex:auerbachcantor}
We present a smooth, strictly convex, centrally symmetric planar body $K$ such 
that $A(K)$ is the union of two disjoint copies of the Cantor set and a 
countable set of isolated points.

First, let $K_0$ be the Euclidean unit disk centered at the origin, and let 
$C$ denote the shorter arc connecting the two points whose angles with the 
positive $x$ axis are $-\pi/4$ and $\pi/4$. Let $C_0$ denote the Cantor set in 
$C$. Now, $C_0$ can be written as 
\[
C_0=C\setminus\bigcup_{n=1}^{\infty}I_n,
\]
where the $I_n$ are disjoint open arcs in $C$. 

For each $n\in\bZ^+$, we construct a smooth and strictly convex curve $C_n$ 
connecting the two endpoints of $I_n$ with the following properties.
\begin{enumerate}
 \item 
$C_n$ has the same tangents at the endpoints as $K_0$;
\item 
$C_n$ is contained in $\conv I_n$; 
\item 
for any point $x$ of $C_n$ which is neither an end point, nor the midpoint of 
$C_n$, the tangent of $C_n$ at $x$ is not orthogonal to $x$.
\end{enumerate}

Let 
\[
 \Psi(x) = 
\begin{cases}
\exp\left( -\frac{1}{1 - x^2}\right) & \mbox{ for } x\in (-1,1) \\
0 & \mbox{ otherwise} 
\end{cases}
\]
denote a \emph{bump function}. It is well known that $\Psi$ is non-negative, 
smooth, its support is $[-1,1]$, and the only points in its support where 
the derivative is zero are $-1,1$ and $1/2$.

Let $\alpha_0< \alpha_1$ denote the angles for which the two end points of 
$I_n$ have polar coordinates $(\alpha_0,1)$ and $(\alpha_1,1)$. Let $C_n$ be 
the curve whose polar coordinates are given for each 
$\varphi\in[\alpha_0,\alpha_1]$ by 
\begin{equation*}
\bigg(\varphi,1-\varepsilon\Psi\left(\frac{2}{\alpha_1-\alpha_0}\left[
\varphi-\frac{\alpha_0+\alpha_1}{2}\right]\right)\bigg),
\end{equation*}
where $\varepsilon>0$ is small.

Clearly, $C_n$ is a smooth curve, and if $\varepsilon$ is small, then it is 
also strictly convex.
Moreover, $C_n$ has the first property, as 
$\Psi^{\prime}(-1)=\Psi^{\prime}(1)=0$. If $\varepsilon$ is small, 
then $C_n$ has the second property as well. Finally, to verify the third 
property, 
observe that the tangent of $C_n$ is orthogonal to $x$ at a point $x$ of $C_n$ 
if and only if, the derivative of 
\[
1-\varepsilon\Psi\left(\frac{2}{\alpha_1-\alpha_0}\left[
\varphi-\frac{\alpha_0+\alpha_1}{2}\right]\right)
\]
as a function of $\varphi$ is zero at the polar angle $\varphi$ of $x$. 
However, 
that is only the case at the two end points and at the mid point of $C_n$.

Consider the following closed curve.
\[
L=\left(C_0\cup -C_0\right)\cup\bigcup_{n=1}^{\infty}\left(C_n\cup -C_n 
\right).
\]
Clearly, $L$ is the boundary of a smooth, strictly convex, centrally symmetric 
planar body, call it $K$. By construction, for each $n\in\bZ^+$, only the 
mid-point of the relative interior of the arc $C_n$ is an Auerbach point. The 
same holds for $-C_n$. On the other hand, all points of $C_0\cup-C_0$ are 
Auerbach points. Thus, $K$ is as promised.
\end{example}

\section{Proof of Proposition~\ref{thm:continuoumeasure}}\label{sec:3}

We may assume that $0,1\in H$. Enumerate the components of $\Re\setminus H$ as 
$I_0,I_1,\ldots$, where $I_0=(-\infty,0)$ and $I_1=(1,\infty)$. We will assign 
recursively a real number $a_n$ to each open interval $I_n$. Let $a_0=0$ and 
$a_1=1$.

If $a_k$ is defined for $k<n$, then let
\begin{equation*}
 a_n=\frac{1}{2}\left(\max_{\ell<n,\, I_{\ell}<I_n} a_{\ell}+\min_{\ell<n, \,
I_{\ell}>I_n} a_{\ell} \right),
\end{equation*}
that is, we consider the two intervals with indices less than $n$ just below 
and just above $I_n$, and $a_n$ is the average of the two values assigned to 
these two intervals.

We define a function $f$ on $\Re$ as follows. First, on $\Re\setminus H$, let 
$f|_{I_n}=a_n$. To extend $f$ to $\Re$, we set
\begin{equation}
 a(x)=\sup (-\infty,x)\setminus H, \text{ and }  
 b(x)=\inf (x,\infty)\setminus H.
\end{equation}

If $x\in H$ and $a(x)=b(x)$, then clearly, the left limit, $f(a(x)-)$, of $f$ 
at $a(x)$ is equal to the right limit $f(b(x)+)$. Thus, the function
\begin{equation*}
 f(x)=
 \begin{cases}
	a_n, \text{ if } x\in I_n;\\
	f(a(x)-)=f(b(x)+), \text{ if }  x\in H  \text{ and }  x=a(x)=b(b);\\
	f(a(x)-)\frac{b(x)-x}{b(x)-a(x)}+ f(b(x)+)\frac{x-a(x)}{b(x)-a(x)}, 
\text{ if }  
x\in H \text{ and } a(x)<b(x)
\end{cases}
\end{equation*}
is continuous, strictly increasing, and is locally constant on $\Re\setminus H$.

Finally, let $\mu_0$ denote the Lebesgue-Stieltjes measure corresponding to 
$f$, and $\mu_1$ the measure $\mu_1(A)=\lambda(A\cap H)$, where $\lambda$ is the 
Lebesgue measure. Clearly, $\mu=\mu_0+\mu_1$ is a continuous measure, and 
$\supp \mu \subseteq H$.

To show the reverse containment, let $I$ be an open interval and assume that 
$I\cap H\neq\emptyset$. If $I\cap H$ is of positive Lebesgue measure, then 
$\mu(I)>\mu_1(I)>0$. Otherwise, $I$ is intersected by at least two $I_k$. 
Indeed, if only one $I_k$ would intersect $I$, then $I\cap H$ would 
be the union of at most two intervals contradicting that $H$ is perfect and 
of Lebesgue measure zero.

Since the values of $f$ on distinct intervals $I_k$ are distinct, $f$ is not 
constant on $I$, and hence, $\mu(I)>\mu_0(I)>0$, completing the proof of 
Proposition~\ref{thm:continuoumeasure}.

\section*{Acknowledgement}
KS thanks Adam Ostaszewski for enlightening conversations and for drawing his attention to \cite{P2005}.

\bibliographystyle{amsalpha}
\bibliography{biblio}
\end{document}